\documentclass{article}

\usepackage{latexsym}
\usepackage{array}
\usepackage{ifthen}

\usepackage{amsmath}
\usepackage{amsfonts,euscript,eufrak,latexsym}
\usepackage{amssymb}
\usepackage{amsthm}
\usepackage{mathrsfs}
\usepackage{stmaryrd}
\usepackage{breqn}
\usepackage{graphicx}
\usepackage{grffile}
\usepackage{subcaption}
\usepackage{pgfplots}
\usepackage{pgfplotstable}
\usepackage{bm}
\usepackage{cite}
\newcommand{\h}[0]{\mathpzc{h}}

\DeclareFontFamily{OT1}{pzc}{}
\DeclareFontShape{OT1}{pzc}{m}{it}{<-> s * [1.10] pzcmi7t}{}
\DeclareMathAlphabet{\mathpzc}{OT1}{pzc}{m}{it}

\numberwithin{equation}{section}

\newcommand{\vertiii}[1]{{\left\vert\kern-0.25ex\left\vert\kern-0.25ex\left\vert #1 
    \right\vert\kern-0.25ex\right\vert\kern-0.25ex\right\vert}}
\newcommand{\Bf}[0]{\mathcal{B}}

\newcommand{\Lf}[0]{\mathcal{L}}

\newcommand{\Ch}[0]{\mathcal{C}_h}
\newcommand{\Eh}[0]{\mathcal{E}_h}
\newcommand{\Gh}[0]{\mathcal{G}_h}

\newcommand{\osc}[0]{\mathrm{osc}}

\numberwithin{equation}{section}

 \newtheorem{theorem}{Theorem}[section]
 
 \newtheorem{lemma}[theorem]{Lemma}

\theoremstyle{definition}

\newtheorem{remark}[theorem]{Remark}

\begin{document}

\title{Nitsche's method for unilateral contact problems}
\author{Tom Gustafsson\thanks{Funding from Tekes (Decision number 3305/31/2015) and the Finnish Cultural Foundation is gratefully acknowledged.}, Rolf Stenberg and Juha Videman\thanks{Work supported by the Portuguese Science Foundation (FCOMP-01-0124-FEDER-029408) and the Finnish Academy of Science and Letters.}}

\maketitle

\begin{abstract}
We derive optimal a priori and a posteriori error estimates for Nitsche's method applied to unilateral contact problems. Our analysis is based on the interpretation of Nitsche's method as a stabilised finite element method for the mixed Lagrange multiplier  formulation of the contact problem wherein the Lagrange multiplier has been eliminated elementwise. To simplify the presentation, we focus on the scalar Signorini problem and outline only the proofs of the main results since most of the auxiliary results can be traced to our previous works on the numerical approximation of variational inequalities. We end the paper by presenting results of our numerical computations which corroborate the efficiency and reliability of the a posteriori estimators.
\end{abstract}





\section{Introduction} Unilateral contact problems  are of great engineering interest and they occur in numerous areas of physics and mechanics; cf. \cite{KS80}.  In mathematical terms, these problems are expressed as  variational inequalities and most often approximated by the finite element method; cf.  \cite{HHN96} and all the references therein.  

One of the most common examples of unilateral contact problems is contact  between two deformable elastic bodies.
Here, we will consider the Signorini problem which consists in finding the equilibrium position of an elastic body resting on a rigid frictionless surface.
For simplicity, we consider the scalar version of such problem, at times referred to as the Poisson-Signorini problem. However, our results carry over, with minor modifications, to the Signorini problem in linear elasticity. 

The finite element treatment of the Signorini problem  has shown to be more difficult than that of the obstacle problem (another archetypical variational inequality) due to the Signorini (no-penetration) condition at the boundary, and has led to a number of  papers over the  past years; cf. \cite{SV77,BHR78,HL80,BB2000,BB03,BBR03,SCH11,HR12,DH15} and the review papers \cite{W11,Frenchsurvey}. The above mentioned works address  primal and mixed formulations and focus on obtaining optimal a priori estimates  based on Falk's Lemma  \cite{F74}. As for the a posteriori error estimates for the Signorini problem, we refer to  \cite{BS2000,HN05,HN07,WW09, SCH09,KVW15}. 

Another approach that, at the same time, imposes the contact boundary conditions weakly, avoids the additional Lagrange multiplier of the mixed formulations and is consistent in contrast to the standard penalty formulations, is the Nitsche's formulation, first proposed for the unilateral contact problems by Chouly and Hild  \cite{CH13}, see also \cite{CHR15, BHL17} for further generalisations.
In \cite{CH13,CHR15, BHL17}, optimal a priori estimates were derived but, to our knowledge, the only existing work on the a posteriori error estimation of unilateral contact problems approximated by Nitsche's method is the recent work by Chouly et al. \cite{CFHPR18}.

In this paper, we will continue to advocate (cf.  \cite{GSV17,GSV17a, GSV17b,GSV18,GSV18b}) that  Nitsche's method is most readily analysed as a stabilised finite element method,
the relation which was first suggested  in \cite{Rolf95}.  We will prove optimal a priori estimates for the stabilised mixed formulation and show the reliability and the efficiency for the a posteriori error estimators without additional saturation assumptions which are needed when  Nitsche's method is analysed directly (cf.  \cite{CFHPR18}).  We do emphasise though  that the stabilised method is best implemented through the Nitsche's formulation.

The paper is organised as follows. In Section 2, we introduce the Poisson-Signorini problem, write it in a mixed variational form and state a continuous stability  result. In Section 3, we formulate the   stabilised finite element method, state a discrete stability estimate and prove an a priori error estimate. In Section 4, we introduce residual-based a posteriori error estimators and establish lower and upper bounds for the error in terms of the estimators. In Section 5, we deduce the Nitsche's formulation from the stabilised one and in Section 6
 report on our numerical computations.  
 
 We note that some of the proofs have been left out since they are formally similar to the ones proven in our previous works, in particular in \cite{GSV17,GSV18}.

\section{The continuous problem}

Let  $\Omega \subset \mathbb{R}^d$, $d\in\{2,3\}$, denote a polygonal (or polyhedral) domain
and $ \partial \Omega $ its boundary with $\partial\Omega =  \Gamma_D \cup \Gamma_N \cup \Gamma$. We assume that the  boundary parts $\Gamma_D,\Gamma_N$ and $\Gamma$ are all non-empty and non-overlapping, $\overline{\Gamma_D} \cap \overline{\Gamma}= \emptyset$ and that the part $\Gamma $ where the contact may occur coincides with one of the sides of the polygon (or the polyhedron). 

The Poisson-Signorini unilateral contact problem can be written as follows: find
$u$ such that
\begin{equation}\label{strong} 
\begin{aligned}
    -\Delta u &= f \quad && \text{in $\Omega$\,,} \\
    u &= 0 && \text{on $\Gamma_D$,} \\
    \frac{\partial u}{\partial n}  &= 0 && \text{on $\Gamma_N$,}  \\
    u\geq 0, \ \ \frac{\partial u}{\partial n}  &\geq 0, \ \  u \, \frac{\partial u}{\partial n} = 0 && \text{on $\Gamma$}\, ,
    \end{aligned}
\end{equation}
where $f \in L^2(\Omega)$ is a given load
function and $n$ denotes the outer normal vector to  $\Omega$. 
\begin{remark}
For  ease of exposition, we consider here  the simplest version of the Signorini problem although we could easily deal with a non-homogeneous Neumann condition $ \frac{\partial u}{\partial n}  = g$ on $\Gamma_N$, with  a gap function $\psi$ defined on $\Gamma$, such that  $u\geq \psi$ and $(u-\psi)\frac{\partial u}{\partial n} = 0 $ on $\Gamma$, and also with the linear elastic Signorini problem. 
    Note also that by assuming that $\overline{\Gamma_D} \cap \overline{\Gamma}= \emptyset$  we avoid introducing the Lions-Magenes space $H^{1/2}_{00}(\Gamma)$; cf. \cite{tartar2007introduction}.
\end{remark}

To give a weak formulation for problem \eqref{strong}, we introduce the Hilbert spaces
\[ 
V =\{ v\in H^{1}(\Omega) \, :\, v=0 \ {\rm  on} \ \Gamma_D \} ,  \quad W:= \{ w|_{\Gamma}: w\in V\}  =H^{1/2}(\Gamma),
\]
endowed with the norms
\[
\|v\|_V = \|\nabla v\|_0 \, , \qquad \|w\|_{W} = \inf_{v\in V, v|_{\Gamma}=w}\| v\|_V \, .
\]
Now, defining a non-negative  Lagrange multiplier by 
\begin{equation}
\lambda = \frac{\partial u}{\partial n} ,
\label{lagrmult}
\end{equation}
 we obtain the following (weak) mixed  variational   formulation of problem  \eqref{strong} (cf. \cite{HHNL}): find
$(u,\lambda)\in V\times \Lambda$ such that
 \begin{equation} \begin{aligned}
(\nabla u, \nabla v) - \langle v,\lambda\rangle &=(f,v) \quad & \forall v\in V \, , \\
\langle u, \mu-\lambda\rangle & \geq 0   \quad &\forall \mu \in \Lambda \, .
    \end{aligned}
\label{dualvf}
\end{equation}
Above
\[
\Lambda  = \{\mu \in Q  : \langle w,\mu\rangle \geq 0 \ \forall w\in W, w \geq 0 \ {\rm a.e. \ on} \ \Gamma \} \, ,
\]
where $Q:=W^\prime$ is the topological dual of $W$ and $\langle\cdot,\cdot\rangle:W\times Q \rightarrow \mathbb{R}$ denotes the duality pairing. The norm in $Q$ is defined as
\[
\|\mu\|_{Q} = \sup_{w\in  W} \frac{ \langle w,\mu\rangle } {\|w\|_W} \, .
\]

\begin{remark}
The corresponding primal  weak formulation of problem \eqref{strong}  reads: find $u\in K$ such that
 \begin{equation}
(\nabla u, \nabla (v-u))\geq (f,v-u) \quad \forall v\in K \, ,
\label{primalvf}
\end{equation}
where 
\[
K = \{ v\in V  : v \geq 0 \ {\rm a.e. \ on} \ \Gamma \} \, .
\]
\end{remark}

\begin{remark}
    \label{bindingremark}
It is well-known that problem \eqref{primalvf}, equivalently \eqref{dualvf}, admits a unique solution $u\in H^1(\Omega)$; cf. \cite{KS80,KO88}. However,  even if the boundary is smooth and the data  regular enough, the Signorini conditions limit the regularity of the solution. In fact, in the two-dimensional case the solution belongs only to $H^{r}, r<5/2$, in the vicinity of points at $\Gamma$ where the constraints change from binding to non-binding, cf. \cite{MM92} and the discussion in \cite{BB03}. If the boundary $\Gamma$ has corners, further singularities may occur; cf. \cite{ABBH13}.
\end{remark} 

Defining  the bilinear and linear forms, $\Bf : (V \times Q) \times (V \times Q) \rightarrow \mathbb{R}$ and $\Lf : V  \rightarrow \mathbb{R}$  through
\[
    \Bf(w,\xi;v,\mu) = (\nabla w, \nabla v) - \langle w, \mu \rangle - \langle v, \xi \rangle, \qquad
    \Lf(v) = (f,v) \, ,
\]
the mixed variational formulation of \eqref{dualvf}
reads: find $(u, \lambda) \in V \times \Lambda$ such that
\begin{equation}
    \label{contprob}
    \Bf(u,\lambda; v, \mu - \lambda) \leq \Lf(v) \quad \forall (v, \mu) \in V \times \Lambda.
\end{equation}

The proof of the following result is straightforward (see, e.g., \cite{GSV17}):
\begin{theorem}[Continuous stability] \label{contstab} For all $(v,\xi)\in V\times Q$ there exists $w\in V$ such that 
\begin{equation}
 \Bf(v,\xi;w,-\xi)\gtrsim   \big(   \Vert v  \Vert_{V} + \Vert \xi  \Vert_{Q}  \big)^{2}
 \end{equation}
  and 
   \begin{equation}
\Vert w\Vert_{1} \lesssim \Vert v  \Vert_{V} + \Vert \xi  \Vert_{Q} .
 \end{equation}
\end{theorem}
Note that above and in the following we write $a \gtrsim b$ (or $a \lesssim b$) when $a \geq C b$ (or $a \leq C b$) for some positive constant $C$ independent of the finite element mesh.

\section{The stabilised method}

Our discretisation is based on a conforming shape-regular  partitioning $\Ch$ of $\Omega$ into non-overlapping triangles or tetrahedra, with $h>0$ denoting the mesh parameter.  We denote the interior edges or facets  of $\Ch$ by $\mathcal{E}_h$  and let $\mathcal{G}_h$ and $\mathcal{N}_h$ be the partitioning of the boundary $\Gamma$ and $\Gamma_N$, respectively, corresponding to the partitioning $\Ch$,   into line segments ($d = 2$) or triangles ($d = 3$).
The finite element spaces 
\[
V_h\subset V, \quad  Q_h\subset Q, \]
are  finite dimensional and consist of piecewise polynomial functions. We also define
\[
 \Lambda_h=\{\mu_h\in Q_h : \mu_h  \geq 0 \ {\rm on} \ \Gamma\}\subset \Lambda\,,
\]
and introduce the discrete bilinear form  $\mathcal{B}_h$ as
\[
   \displaystyle  \mathcal{B}_h(w,\xi;v,\mu)= \mathcal{B}(w,\xi;v,\mu) - \alpha\, \mathcal{S}_h(w,\xi;v,\mu)\,,\\
\]
where $\alpha>0$ is a stabilisation parameter and  the stabilising term $\mathcal{S}_h$ is defined as 
\[
\mathcal{S}_h(w,\xi;v,\mu)= \sum_{E\in \mathcal{G}_h} h_E \left(\xi -\frac{\partial w}{\partial n},\mu -\frac{\partial v}{\partial n}\right)_E ,
\]
with  $h_E$ denoting the diameter of $E\in\mathcal{G}_h$.  

The
stabilised  finite element method is now formulated as:
find $(u_h,\lambda_h)\in V_h\times \Lambda_h$ such that
\begin{equation}
\mathcal{B}_h(u_h,\lambda_h;v_h,\mu_h-\lambda_h) \leq \mathcal{L}(v_h)  \quad \forall(v_h,\mu_h)\in V_h\times \Lambda_h.
\label{stabfem}
\end{equation}

The proof of stability for method \eqref{stabfem} is very similar to the one given for the obstacle problem in \cite[Theorem 4.1]{GSV17}, see also \cite{GSV18}, and is thus omitted. We only note that to estimate the stabilising term we use the following inverse estimate, proven by a scaling argument.
    \begin{lemma}
    There exists $C_I > 0$, independent of $h$, such that
    \begin{equation}
        C_I \sum_{E \in \Gh} h_E \left\| \frac{\partial v_{h}}{\partial n} \right\|_{0,E}^2 \leq  \|v_{h}\|_{V}^2 \quad \forall v_{h} \in V_{h}\, .
    \label{inverse}
    \end{equation}
    \end{lemma}
    
\begin{theorem}[Discrete stability] \label{discrstab} Let $ 0< \alpha < C_{I}$.
For all $(v_{h},\xi_{h})\in V_{h}\times Q_{h}$, there exists $w_{h}\in V_{h}$, such that 
\begin{equation}
 \Bf_{h}(v_{h},\xi_{h};w_{h},-\xi_{h}) \gtrsim \left(   \| v _{h}\|_V+ \| \xi_{h} \|_{Q}   + \left(\sum_{E \in \Gh} h_E \left\|\xi_h\right\|_{0,E}^2\right)^{1/2} \right)^2
\end{equation}
and 
\begin{equation}
  \|w _{h}\|_{V}\lesssim \| v _{h}\|_{V}+ \| \xi _{h}\|_{Q}.
\end{equation}
\label{discstab}
\end{theorem}

\begin{remark}
The additional stability for $\xi_h$   in a mesh-dependent norm in Theorem~\ref{discstab}  will be needed in the proof of Theorem \ref{apriori}. 
\end{remark}

We will need the following Lemma in establishing the a priori error estimate. Its detailed proof can be found in \cite{GSV18}.
\begin{lemma}
    \label{lem:lowresidual} Let $f_h\in V_h$ be the $L^2$ projection of $f$, define
\[
\osc_K(f)=h_K\| f-f_h\|_{0,K}
\]
and, for each $E \in \Gh$, let  $K(E) \in \Ch$ denote the element satisfying $\partial K(E) \cap E = E$.
    For any $(v_h, \mu_h) \in V_h \times Q_h$ it holds that
    \begin{equation}
        \label{eq:lowresidual}
        \begin{aligned}
            &\left( \sum_{E \in \Gh} h_E \left\| \mu_h - \frac{\partial v_{h}}{\partial n} \right\|_{0,E}^2\right)^{1/2} \\
            &\qquad \lesssim \|u-v_h\|_V  + \|\lambda-\mu_h\|_Q + \left(\sum_{E \in \Gh} {\rm osc}_{K(E)} (f)^2\right)^{1/2}.
        \end{aligned}
    \end{equation}
\end{lemma}

As usual, the a priori  estimate now follows from the discrete stability estimate.

\begin{theorem}[A priori error estimate] Let $(u, \lambda) \in V \times \Lambda$ be the solution to the continuous problem \eqref{contprob} and let $(u_h,\lambda_h)\in V_h\times \Lambda_h$ be its approximation obtained by solving  the discrete problem \eqref{stabfem}.
Then the following  estimate holds
\[
        \begin{aligned}
\|u-u_h\|_V+\| \lambda -\lambda_h\|_{Q} \  \lesssim \ & \inf_{v_h\in V_h} \| u-v_h\|_V + \inf_{\mu_h\in\Lambda_h} \left( \| \lambda-\mu_h\|_{Q} + \sqrt{\langle u,\mu_h\rangle}\, \right) \\
            &+ \left(\sum_{E \in \Gh} {\rm osc}_{K(E)} (f)^2\right)^{1/2} \, .
      \end{aligned}
      \]
      \label{apriori}
\end{theorem}
\begin{proof} In view of Theorem \ref{discrstab}, it holds that
    \begin{align*}
        \Bigg( \left(\sum_{E \in \Gh} h_E \left\|\mu_h -\lambda_h\right\|_{0,E}^2\right)^{1/2}  +   &   \|u_h-v_h\|_V+\|\lambda_h-\mu_h\|_{Q} \Bigg)^2 \\
& \lesssim \mathcal{B}_h(u_h-v_h,\lambda_h-\mu_h;w_h,\mu_h-\lambda_h)\, ,
    \end{align*}
with some $w_h\in V_h$    satisfying
   \begin{equation}
        \|w_h\|_V \lesssim \|u_h-v_h\|_V+\|\lambda_h-\mu_h\|_{Q} \, .
    \label{wbound}
    \end{equation}
    It follows that
    \begin{align*}
        &\mathcal{B}_h(u_h-v_h,\lambda_h-\mu_h;w_h,\mu_h-\lambda_h)\\
        &=\mathcal{B}_h(u_h,\lambda_h;w_h,\mu_h-\lambda_h)-\mathcal{B}_h(v_h,\mu_h;w_h,\mu_h-\lambda_h)\\
                &\leq \mathcal{L}(w_h)+\mathcal{B}(u-v_h,\lambda-\mu_h;w_h,\mu_h-\lambda_h)-\mathcal{B}(u,\lambda;w_h,\mu_h-\lambda_h)\\
        &\phantom{=}+\alpha \sum_{E \in \Gh} h_E \left(\mu_h -\frac{\partial v_h}{\partial n},\mu_h -\lambda_h-\frac{\partial w_h}{\partial n}\right)_E ,\\
    \end{align*}
where we have used the bilinearity of $\mathcal{B}$ and $\mathcal{B}_h$, and the discrete problem statement  \eqref{stabfem}.
Observe that
\[
\begin{aligned}
    \mathcal{L}(w_h) -\mathcal{B}(u,\lambda;w_h,\mu_h-\lambda_h) &= (f,w_h) - (\nabla u, \nabla w_h) + \langle w_h, \lambda \rangle + \langle u,\mu_h-\lambda_h\rangle \\
 & \leq  \langle u,\mu_h-\lambda_h\rangle + \langle u,\lambda_h-\lambda\rangle  =   \langle u,\mu_h-\lambda\rangle =  \langle u,\mu_h\rangle  \, ,
 \end{aligned}
 \]
 where we have used the problem \eqref{dualvf} and recalled that  $0\leq \langle u,\mu-\lambda\rangle,  \ \forall \mu \in \Lambda$.
 Moreover
 \[\begin{aligned}
\sum_{E \in \Gh}  & h_E \left(\mu_h -\frac{\partial v_h}{\partial n},  \mu_h -\lambda_h-\frac{\partial w_h}{\partial n}\right)_E   \\
\leq  \ & \left( \sum_{E \in \Gh} h_E \left\|\mu_h -\frac{\partial v_h}{\partial n}\right\|_{0,E}^2\right)^{1/2} \,\left( \sum_{E \in \Gh} h_E \left\|\mu_h -\lambda_h\right\|_{0,E}^2\right)^{1/2} 
\\
 & +  \left( \sum_{E \in \Gh} h_E \left\|\mu_h -\frac{\partial v_h}{\partial n}\right\|_{0,E}^2\right)^{1/2} \,\left( \sum_{E \in \Gh} h_E \left\|\frac{\partial w_h}{\partial n}\right\|_{0,E}^2\right)^{1/2} \, .
\end{aligned}
\]
The a priori estimate now follows from the continuity of the bilinear form $\mathcal{B}$, the inverse estimate \eqref{inverse} and bound \eqref{wbound}, Lemma \ref{lem:lowresidual}, and from the triangle inequality.
\end{proof}

\section{A posteriori error analysis}

We will define  the local error estimators corresponding to the finite element solution $(u_h,\lambda_h)$ as
\begin{alignat*}{2}
    \eta_K^2 &= h_K^2 \| \Delta u_{h} + f \|_{0,K}^2, \qquad &K \in \Ch, \\
    \eta_{E,\Omega}^2 &= h_E \left \| \left\llbracket  \nabla u_h \cdot n \right\rrbracket \right\|_{0,E}^2, &E \in \Eh, \\
    \eta_{E,\Gamma}^2 &= h_E \left\| \lambda_h - \frac{\partial u_{h}}{\partial n} \right\|_{0,E}^2,  &E \in \Gh ,\label{localest} \\
    \eta_{E,\Gamma_N}^2 &= h_E \left\| \frac{\partial u_{h}}{\partial n} \right\|_{0,E}^2,  &E \in \mathcal{N}_h ,
\end{alignat*}
where $ \left\llbracket  \nabla u_h \cdot n \right\rrbracket $ denotes the  jump in  the normal  derivative across the inter-element boundaries.
  
We also define the global error estimators $\eta$ and $S$ through
\begin{align*}
    \eta^2 & = \sum_{K \in \Ch} \eta_K^2 + \sum_{E \in \Eh} \eta_{E,\Omega}^2 +  \sum_{E \in \Gh} \eta_{E,\Gamma}^2 + \sum_{E \in \mathcal{N}_h} \eta_{E,\Gamma_N}^2  ,  \\
    S  & = \|u_h^{-}\|_W + \sqrt{ (\lambda_h,u_h^{+})}\, .
\end{align*}
where $w_{+}=\max\{w,0\}$  denotes the positive and $w_{-}=\min\{w,0\}$ the negative part of $w$. We will now establish both the efficiency and the reliability of the error estimators.
\begin{theorem}[A posteriori estimate] It holds that
    \begin{equation}\label{reli}
\|u-u_h\|_V + \|\lambda-\lambda_h\|_Q\  \lesssim\,\eta + S
    \end{equation}
     and
      \begin{equation}\label{effi}
           \eta \  \lesssim \  \|u-u_h\|_V + \|\lambda-\lambda_h\|_Q  + \Bigg( \sum_{K \in \Ch} \osc_{K} (f)^2\Bigg)^{1/2}.
    \end{equation}
\end{theorem}
\begin{proof} The continuous stability condition (cf. Lemma \ref{contstab})  guarantees the existence of $v\in V$ such that 
\begin{equation}
\|v\|_V \lesssim   \|u-u_h\|_V +\|\lambda-\lambda_h\|_Q 
\label{vbound}
\end{equation}
and 
\[
\Big( \|u-u_h\|_V +\|\lambda-\lambda_h\|_Q\Big)^2 \lesssim \mathcal{B}(u-u_h,\lambda-\lambda_h;v, \lambda_h-\lambda) .
\]

On the other hand, testing in \eqref{stabfem} with $(-\tilde{v}, \lambda_h)$, where $\tilde{v}$ is the Cl\'ement interpolant of $v$, we obtain
\[
0\leq  - \mathcal{B}_h(u_h,\lambda_h,-\tilde{v},0)+\mathcal{L}(-\tilde{v}) \, .
\]
It follows that
\begin{align*}
    &\Big( \|u-u_h\|_V +\|\lambda- \lambda_h\|_Q\Big)^2  \\ &\lesssim  
\mathcal{B}(u,\lambda;v, \lambda_h-\lambda) - \mathcal{B}(u_h,\lambda_h;v, \lambda_h-\lambda) \\ &  \quad - \mathcal{B}(u_h,\lambda_h,-\tilde{v},0)  +\mathcal{L}(-\tilde{v}) +
\alpha\,\mathcal{S}_h(u_h,\lambda_h,-\tilde{v},0) \\
    \qquad \qquad &  \lesssim   \mathcal{L}(v-\tilde{v}) - \mathcal{B}(u_h,\lambda_h;v-\tilde{v}, \lambda_h-\lambda) +\alpha\,\mathcal{S}_h(u_h,\lambda_h,-\tilde{v},0) ,
\end{align*}
where we have used the bilinearity of $\mathcal{B}$ and the problem statement \eqref{contprob}. 

Observe  first that
\begin{equation}
\begin{array}{ll}
    \mathcal{L}(v-\tilde{v})  -    \mathcal{B}(u_h,\lambda_h;v-\tilde{v},   \lambda_h-\lambda) \vspace{2mm} \\  \qquad \qquad = (f,v-\tilde{v}) - (\nabla u_h, \nabla(v-\tilde{v})) + \langle v-\tilde{v}, \lambda_h\rangle +\langle u_h, \lambda_h-\lambda\rangle .
 \label{apostterms1}
 \end{array}
\end{equation}
After elementwise integration by parts, the first four terms on the right-hand side of \eqref{apostterms1} read  
\begin{align*}
    &\sum_{K\in \Ch} (\Delta u_h+f, v-\tilde{v})_K - \sum_{E\in \Eh}  ( \left\llbracket  \nabla u_h \cdot n \right\rrbracket , v-\tilde{v} )_E \\
    &\quad + \sum_{E\in \Gh} \left( \lambda_h-\frac{\partial u_h}{\partial n} ,  v-\tilde{v} \right)_E - \sum_{E\in \mathcal{N}_h} \left( \frac{\partial u_h}{\partial n} ,  v-\tilde{v} \right)_E   \, .
\end{align*}
The last term  in \eqref{apostterms1} can be estimated as follows
\begin{align*}
\langle u_h, \lambda_h-\lambda\rangle  \leq  (\lambda_h,u_h^{+}) +  \langle u_h^{-}, \lambda_h-\lambda \rangle  \leq  (\lambda_h,u_h^{+})  + \|\lambda-\lambda_h\|_Q \|u_h^{-}\|_W \, ,
\end{align*}
given that $\langle u_h^{+}, \lambda \rangle \geq 0$.  For the stabilisation term $ \mathcal{S}_h(u_h,\lambda_h,-\tilde{v},0) $, we obtain 
\begin{align*}
    \sum_{E\in \mathcal{G}_h} &h_E \left(\lambda_h -\frac{\partial u_h}{\partial n},\frac{\partial \tilde{v}}{\partial n}\right)_E \\
    & \leq 
 \left( \sum_{E \in \Gh} h_E \left\|\lambda_h -\frac{\partial u_h}{\partial n}\right\|_{0,E}^2\right)^{1/2} \,\left( \sum_{E \in \Gh} h_E \left\| \frac{\partial \tilde{v}}{\partial n}\right\|_{0,E}^2\right)^{1/2} \,  \\ 
 & \lesssim    \left( \sum_{E \in \Gh} h_E \left\|\lambda_h -\frac{\partial u_h}{\partial n}\right\|_{0,E}^2\right)^{1/2}  \Vert  \tilde{v} \Vert_{V} ,
\end{align*}
where in the last step we have used the discrete inverse estimate \eqref{inverse}.

For the Cl\'ement interpolant  $\tilde{v}$, it holds that
    \begin{equation}\label{cle}
        \Vert  \tilde{v} \Vert_{V}^2+ \sum_{K \in \Ch} h_K^{-2} \Vert v-\widetilde{v}\Vert _{0,K }^2+   \sum_{E \in \Eh \cup \Gh \cup \mathcal{N}_h} h_E^{-1} \Vert v-\tilde{v} \Vert_{0,E}^2  \ \lesssim   \Vert  v \Vert_{V}^2 .
    \end{equation}
The reliability estimate \eqref{reli} can now be established  using the Cauchy-Schwarz inequality together with \eqref{vbound} and \eqref{cle}. 

    The efficiency follows from standard  lower bounds \cite{Vbook} and from Lemma~\ref{inverse}.
\end{proof}

\section{Nitsche's method}

Nitsche's formulation can be elegantly derived from the stabilised method following the reasoning suggested in \cite{Rolf95}. 
 In fact, testing with $(0,\mu_h)$ in \eqref{stabfem}, yields 
 \[
 (u_h,\mu_h-\lambda_h) +\alpha\, \sum_{E\in \Gh} h_E \left( \lambda_h-\frac{\partial u_h}{\partial n} ,  \mu_h-\lambda_h \right)_E \geq  0  \qquad \forall \mu_h\in\Lambda_h \, .
 \]
 In particular,
 \[
 (u_h,\lambda_h) +\alpha\, \sum_{E\in \Gh} h_E \left( \lambda_h-\frac{\partial u_h}{\partial n} ,  \lambda_h \right)_E =  0\, .
 \]
We can thus write
\[
\sum_{E\in \Gamma_C}  (u_h,\mu_h)_E +\alpha\, \sum_{E\in \Gamma_C} h_E \left( \lambda_h-\frac{\partial u_h}{\partial n} ,  \mu_h \right)_E =  0  \qquad  \forall \mu_h\in\Lambda_h ,
 \]
 where  $\Gamma_C$ denotes the contact boundary. Locally, we obtain
 \begin{equation}
 \lambda_h|_{E} = \Pi_h \frac{\partial u_h}{\partial n}\Big|_E - (\alpha h_E)^{-1}  \Pi_h u_h\Big|_E  \qquad \forall E\in \Gamma_C , 
 \label{LMN}
 \end{equation}
 where $\Pi_h$ is the $L^2$ projection onto $\Lambda_h$. Assuming equal  polynomial order for both variables, say $k\geq 1$, but with the discrete Lagrange multiplier being discontinuous from element to  element, it follows that $\Pi_h=I$ and we have
 \[
 Q_h=\{ \mu_h\in Q : \mu_h\in P_k(E)~ \forall E\in \Gh\}  .
 \]
Therefore, using  the test function $(v_h,0)$  in \eqref{stabfem} and substituting \eqref{LMN}  into the resulting equation, leads to 
\begin{align*}
(\nabla  u_h,\nabla v_h) +  & \sum_{E\in \Gamma_C}  (\alpha h_E)^{-1} (u_h , v_h)_E -    \sum_{E\in \Gamma_C} \left\{  \left( \frac{\partial u_h}{\partial n},v_h\right)_E + \left(u_h,   \frac{\partial v_h}{\partial n}\right)_{E} \right\}    \\ 
& -\alpha \sum_{E\in \Gamma \setminus\Gamma_C}  h_E \left(  \frac{\partial u_h}{\partial n} ,  \frac{\partial v_h}{\partial n} \right)_E  = (f, v_h)  \qquad \forall v_h \in V_h . 
\end{align*}
 
 Defining an $L^2(\Gamma)$ function $\h$ through
 \[
 \h|_E = h_E \quad \forall E\in\Gh ,
 \]
 the discrete Lagrange multiplier can be written globally as 
\begin{equation}
    \label{lambda}
 \lambda_h =  \left(\frac{\partial u_h}{\partial n}\Big|_{\Gamma} -   (\alpha \h)^{-1} u_h|_\Gamma\right)_+ 
\end{equation}
and the discrete contact region becomes
\[
\Gamma^h_C= \{ (x,y)\in \Gamma : \lambda_h(x,y) > 0\}. 
\]
{\em Nitsche's method} now reads: find $u_h\in V_h$ and $\Gamma^h_C=\Gamma_C^h(u_h)$ such that 
\begin{equation}\begin{array}{l}
\displaystyle (\nabla  u_h,\nabla v_h)_\Omega +   \ \alpha^{-1} \big( \h^{-1} u_h , v_h\big)_{\Gamma_C^h }-     \left( \frac{\partial u_h}{\partial n},v_h\right)_{\Gamma_C^h } - \left(u_h,   \frac{\partial v_h}{\partial n}\right)_{\Gamma_C^h }   \vspace{2mm} \\ 
\displaystyle \qquad \qquad \qquad  -\alpha  \left(  \h \,  \frac{\partial u_h}{\partial n} ,  \frac{\partial v_h}{\partial n} \right)_{\Gamma \setminus\Gamma_C^h }  = (f, v_h)  \qquad \forall v_h \in V_h . 
\end{array}
\label{nitsche}
\end{equation}

\begin{remark}
Note that  formulation  \eqref{nitsche} is equivalent to the Nitsche's method introduced in \cite{CH13}, with $\gamma = \alpha \h $, and  to the symmetric variant 
$(\theta_1=-1)$ of the Nitsche's method proposed in \cite{CHR15}. 
\end{remark}

For the implementational aspects of the Nitsche's  method \eqref{nitsche}, we refer to \cite{GSV17a} where similar method was applied to the obstacle problem.

\section{Numerical verification}

We consider the problem~\eqref{nitsche} where the domain is given by $\Omega = (0,1)^2$, the loading is $f(x,y) = x \cos(2 \pi y)$, and the boundaries are
\begin{align*}
    \Gamma_D &= \{ (x,y) \in \mathbb{R}^2 : x = 0,~0 < y < 1 \}, \\
    \Gamma_N &= \{ (x,y) \in \mathbb{R}^2 : 0 < x < 1,~y = 0\} \cup \{ (x,y) \in \mathbb{R}^2 : 0 < x < 1,~y = 1\}, \\
    \Gamma &= \{ (x,y) \in \mathbb{R}^2 : x=1,~0 < y < 1 \}.
\end{align*}
The contact region is now a proper subset of $\Gamma$. There are two points on $\Gamma$
where the constraints change from binding to non-binding and the exact solution
belongs to $H^r$, $r<5/2$, see the comments in Remark~\ref{bindingremark}.
We employ the quadratic finite element space
\[
    V_h = \{ w \in V : w|_K \in P_2(K) ~\forall K \in \mathcal{C}_h \}.
\]
Thus, in view of the regularity of $u$, the convergence rate of the error
$\|u-u_h\|_V$~with uniform mesh refinement is limited to
$\mathcal{O}(N^{-3/4})$ where $N$~is the number of degrees of freedom.

To overcome the limited regularity, we consider also
a sequence of adaptive meshes. Starting with an initial mesh,
we compute the discrete solution $u_h$, the corresponding Lagrange multiplier through \eqref{lambda}, and the respective error indicator
\begin{align*}
    \mathcal{E}_K^2 &=  h_K^2 \| \Delta u_h + f \|_{0,K}^2 + h_K \| \llbracket \nabla u_h \cdot n \rrbracket \|_{0,\partial K}^2 \\
    &\quad+ h_K \left\|\lambda_h - \frac{\partial u_h}{\partial n} \right\|_{0,\partial K \cap \Gamma}^2 + h_K \left\|\frac{\partial u_h}{\partial n}\right\|_{\partial K \cap \Gamma_N}^2
\end{align*}
for every $K$ in the initial mesh. The mesh is then
improved by splitting the elements that have large error indicators.
This process is performed repeatedly until  a predetermined value of $N$ is reached.
Please refer to \cite{GSV18} or \cite{GSV18b} for more details on how to choose and split the
elements.

The discrete solution is visualised in Figure~\ref{fig:solution}.
The resulting sequence of adaptive meshes is given in Figure~\ref{fig:meshes}
and the convergence of the global error estimator $\eta+S$ is compared between the uniform and adaptive
mesh sequences in Figure~\ref{fig:convergence}. As expected, the adaptive meshes are
more refined near the singular points on the contact boundary. Moreover, the
observed convergence rate of the uniform refinement strategy is suboptimal, due to the limited regularity
of the exact solution, whereas the adaptive method successfully recovers the optimal
rate of convergence, $\mathcal{O}(N^{-1})$.

\begin{figure}[h!]
    \centering
    \includegraphics[width=0.9\textwidth]{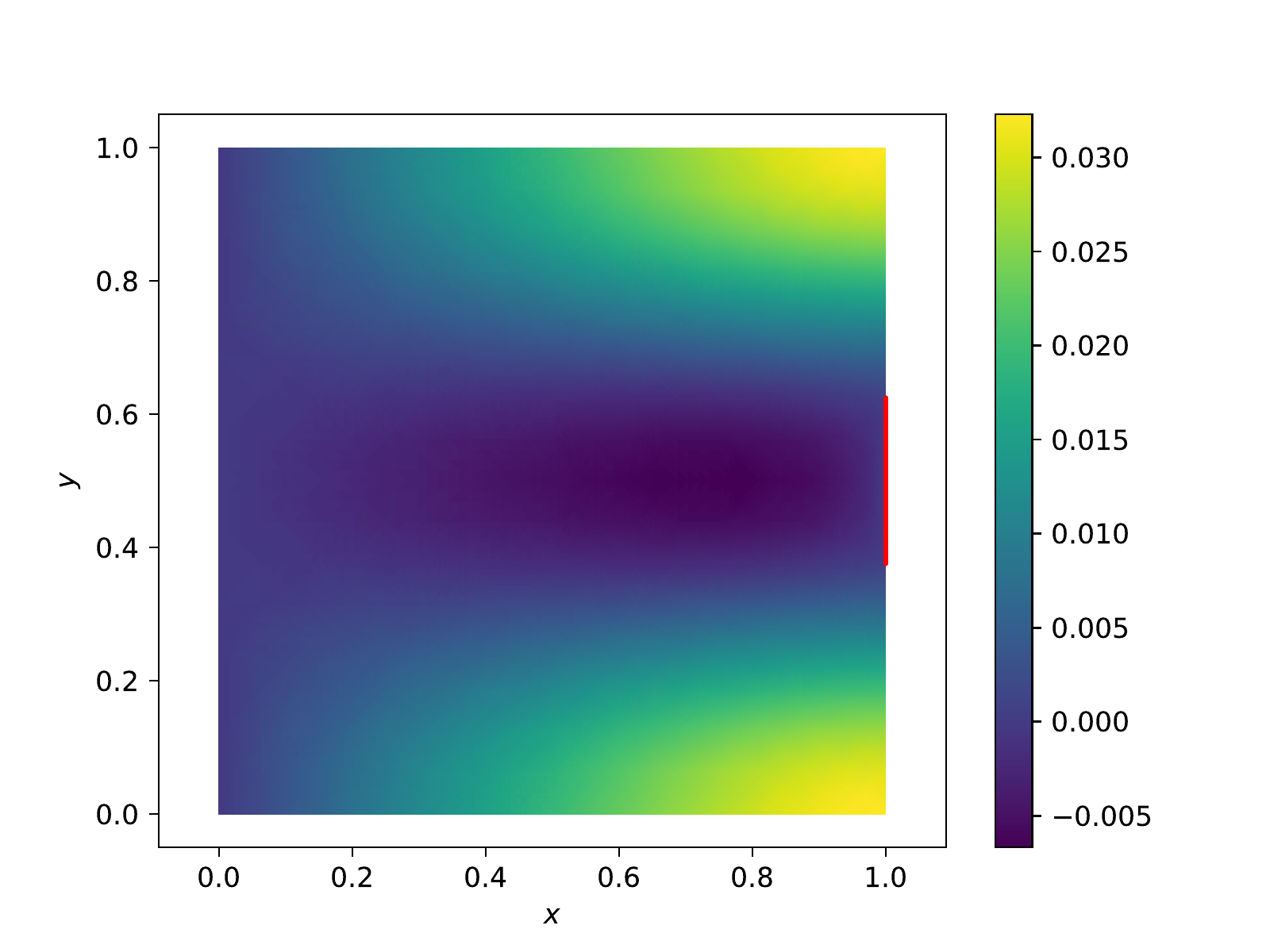}
    \caption{The discrete solution after 8 adaptive refinements. The contact
    region on the rightmost boundary  is highlighted in red.}
    \label{fig:solution}
\end{figure}

\begin{figure}[h!]
    \centering
    \includegraphics[width=0.23\textwidth]{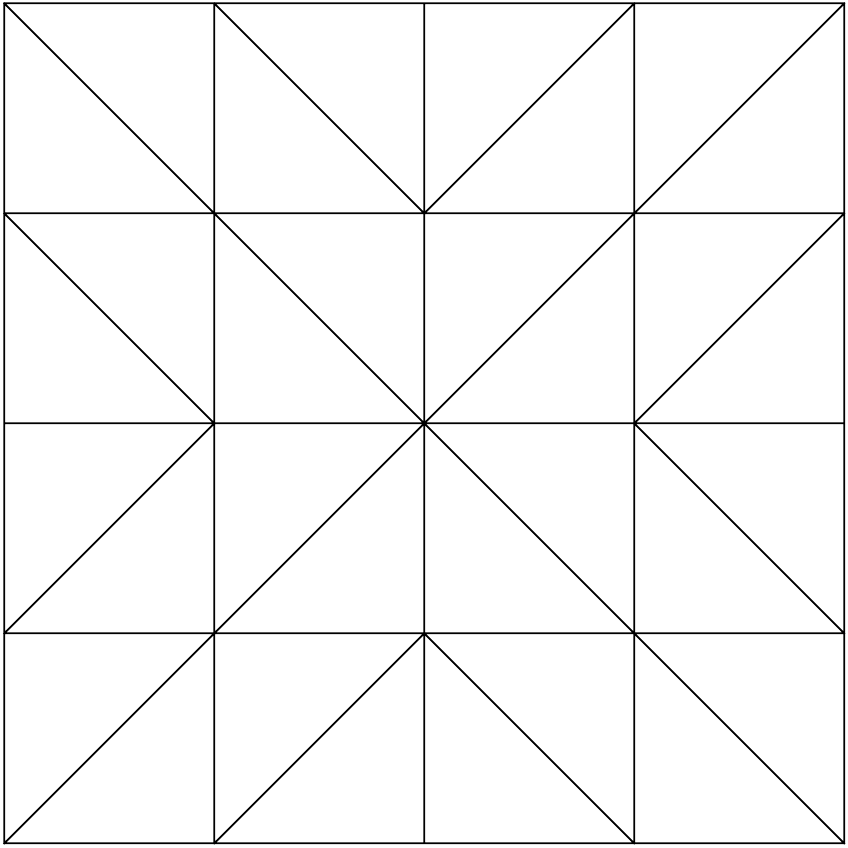}
    \includegraphics[width=0.23\textwidth]{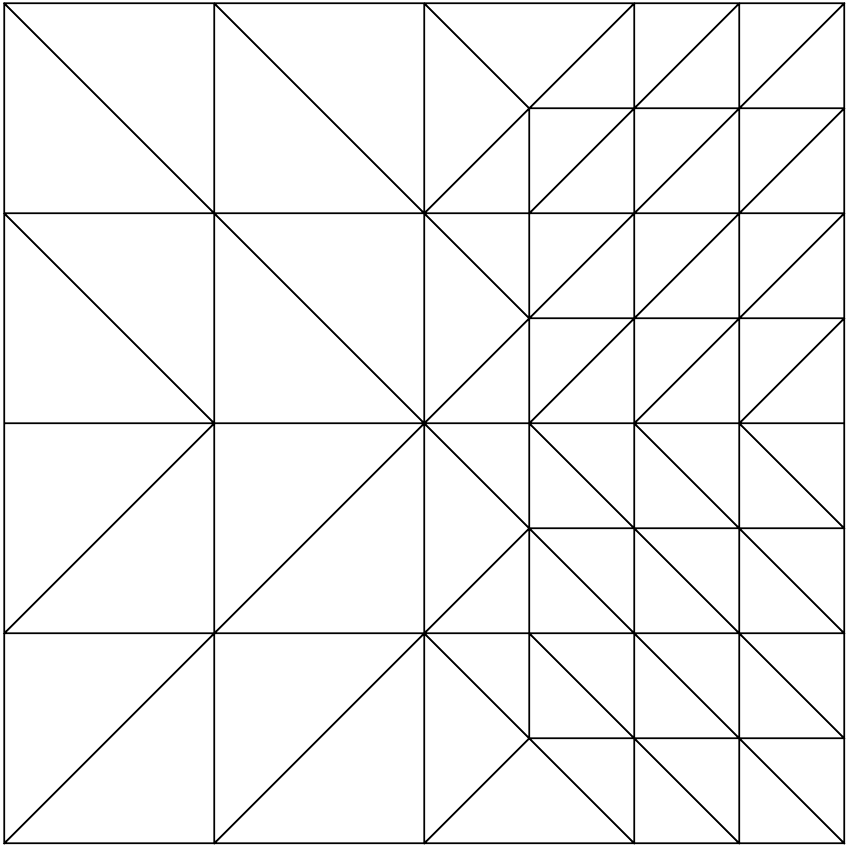}
    \includegraphics[width=0.23\textwidth]{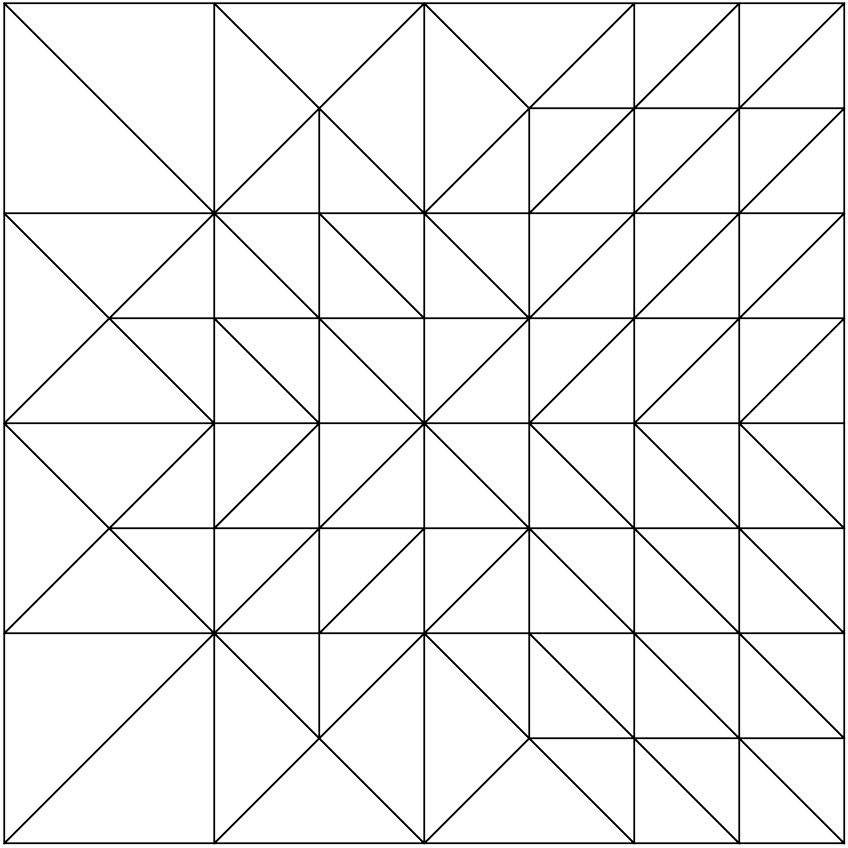}
    \includegraphics[width=0.23\textwidth]{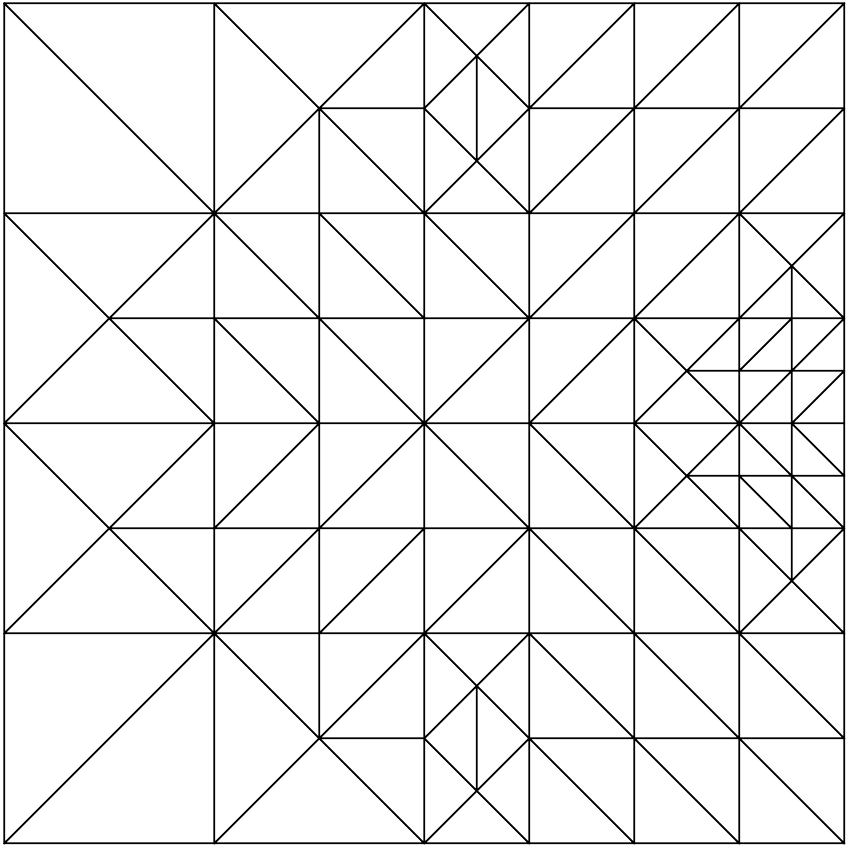}\\[0.1cm]
    \includegraphics[width=0.23\textwidth]{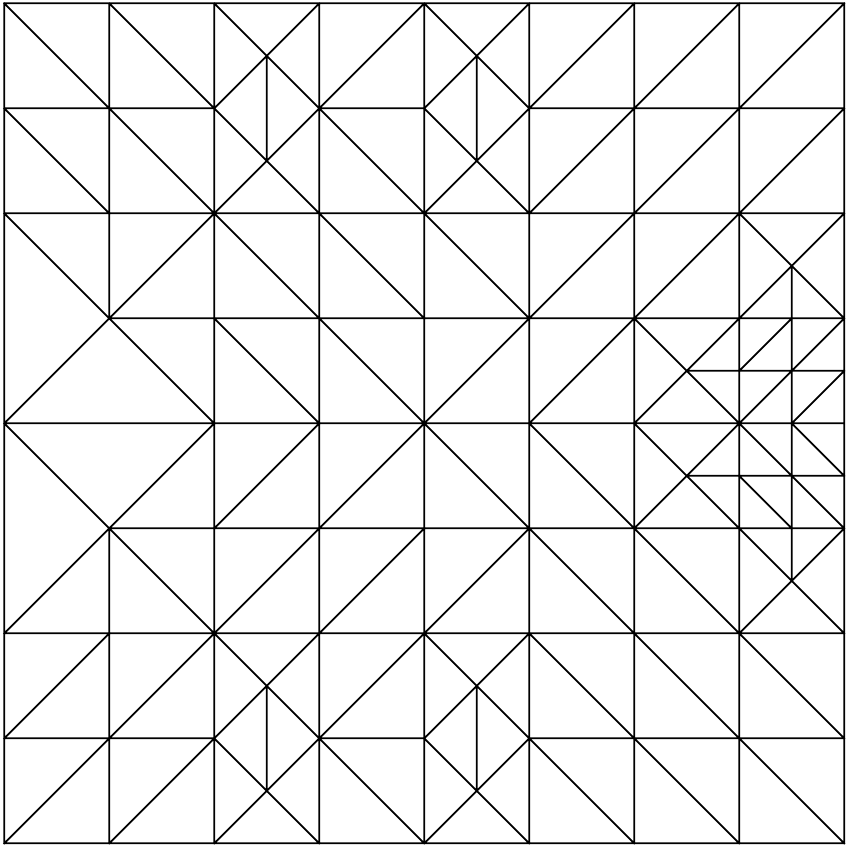}
    \includegraphics[width=0.23\textwidth]{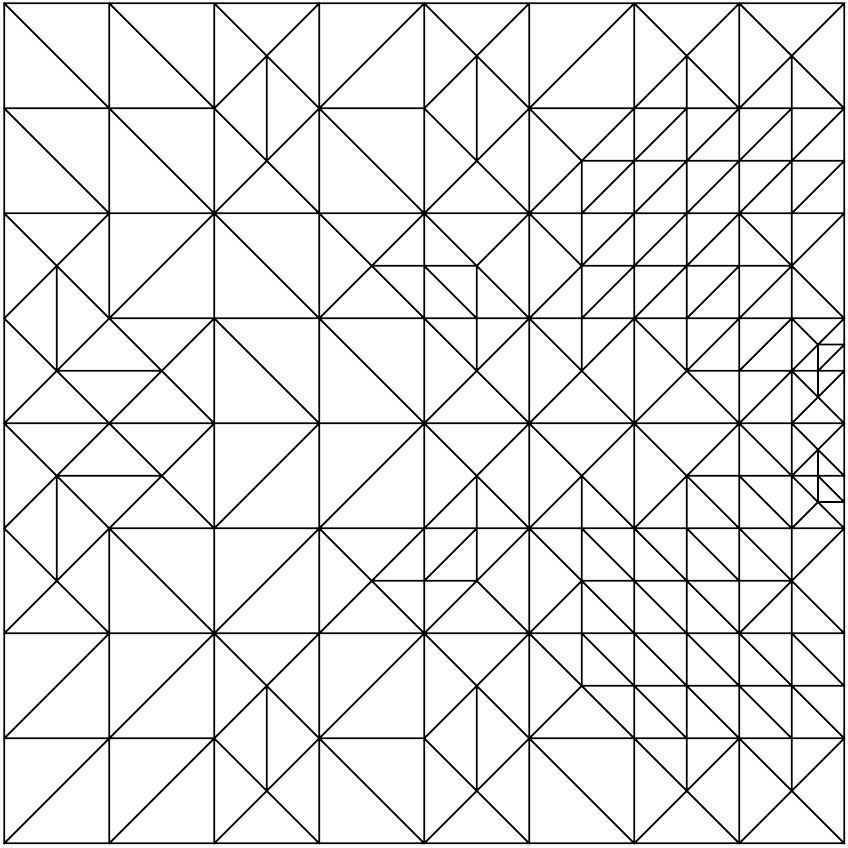}
    \includegraphics[width=0.23\textwidth]{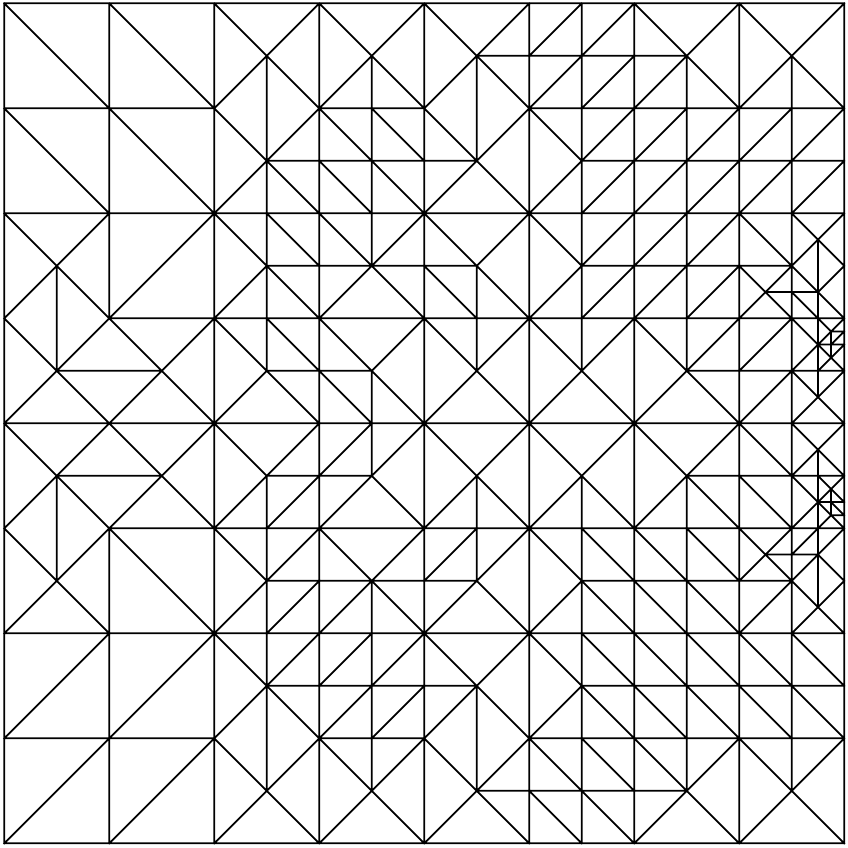}
    \includegraphics[width=0.23\textwidth]{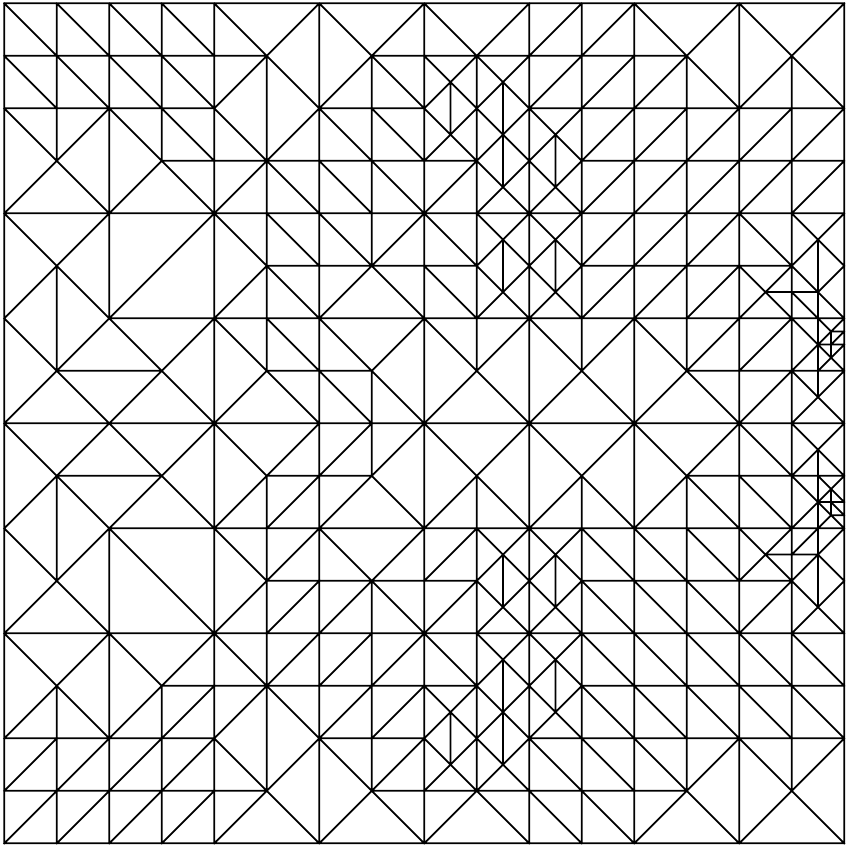}\\[0.1cm]
    \includegraphics[width=0.23\textwidth]{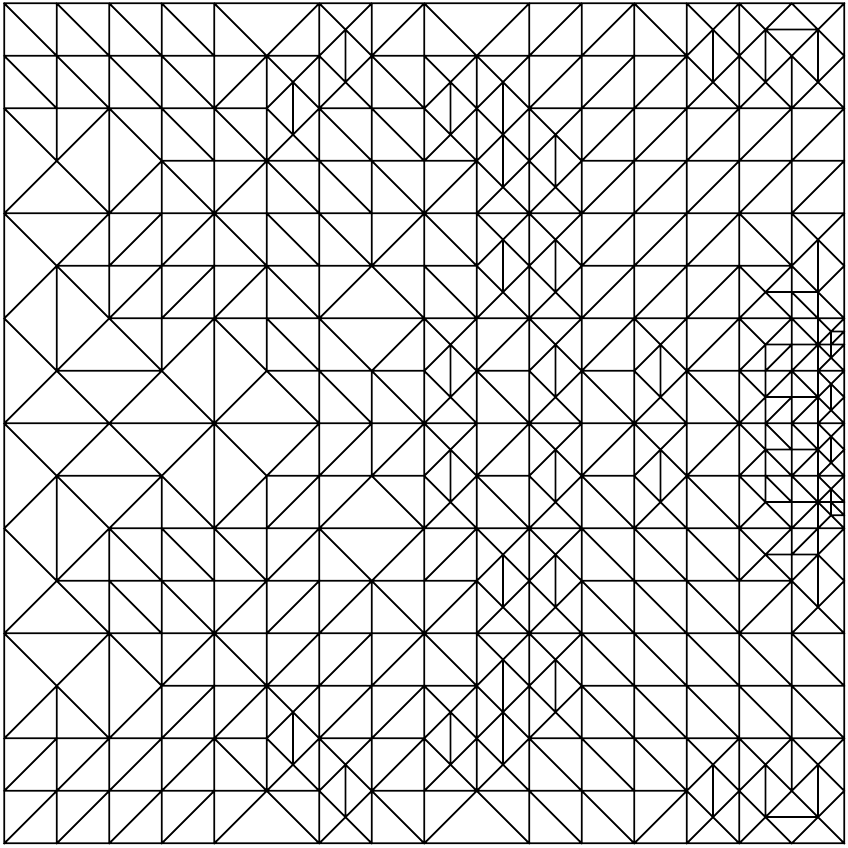}
    \includegraphics[width=0.23\textwidth]{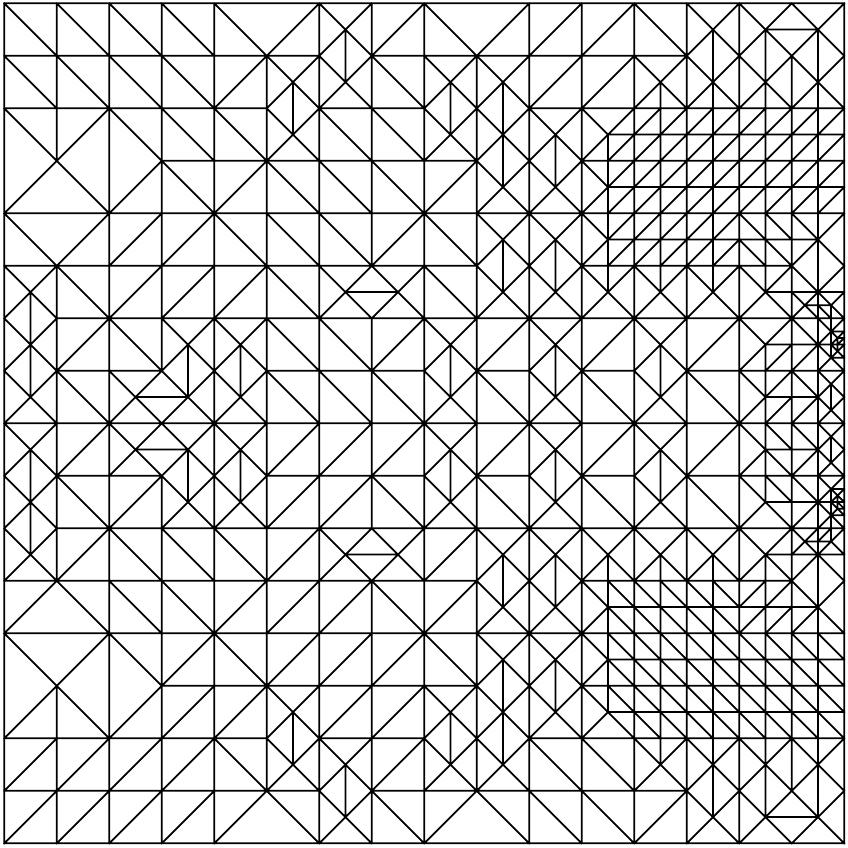}
    \includegraphics[width=0.23\textwidth]{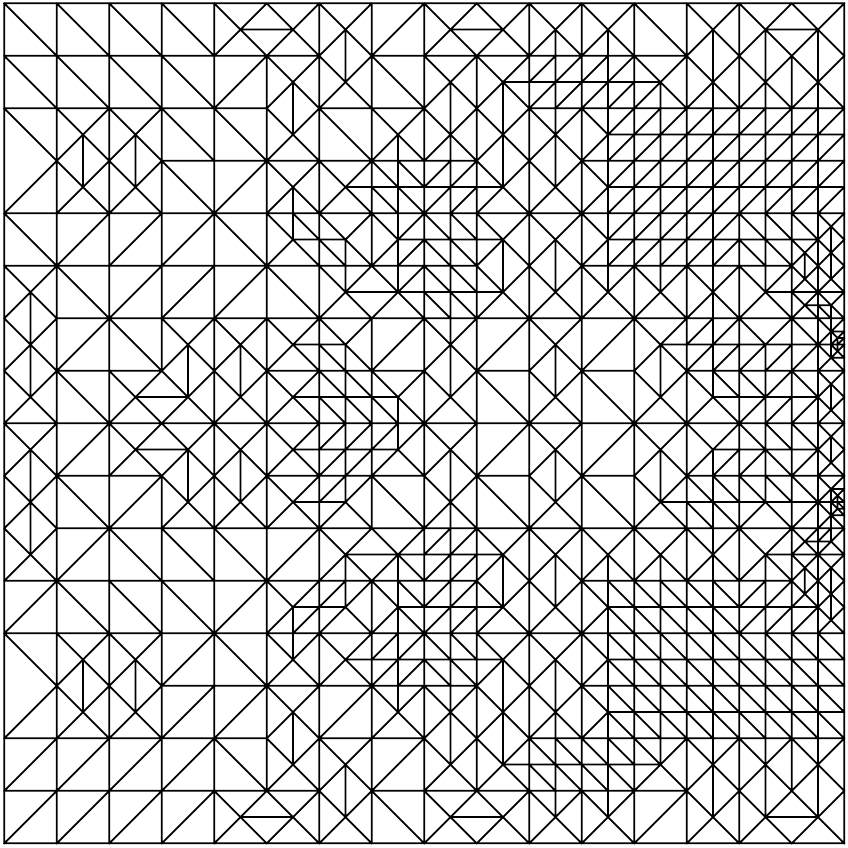}
    \includegraphics[width=0.23\textwidth]{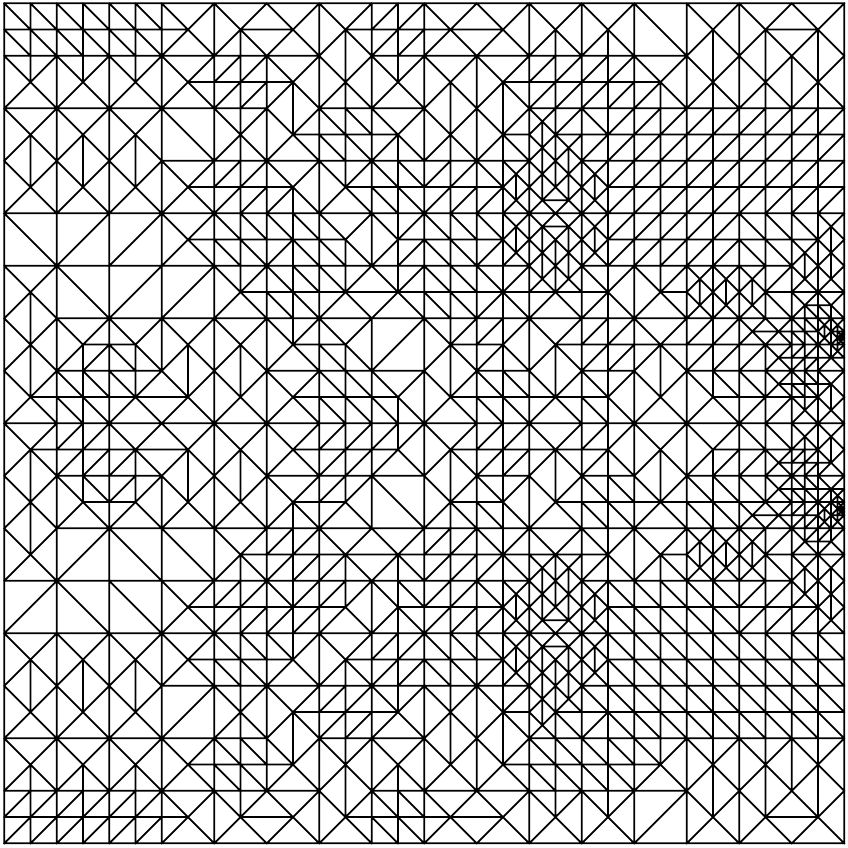}
    \caption{The initial mesh and the resulting sequence of adaptively
    refined meshes.}
    \label{fig:meshes}
\end{figure}

\begin{figure}[h!]
    \centering
    \includegraphics[width=\textwidth]{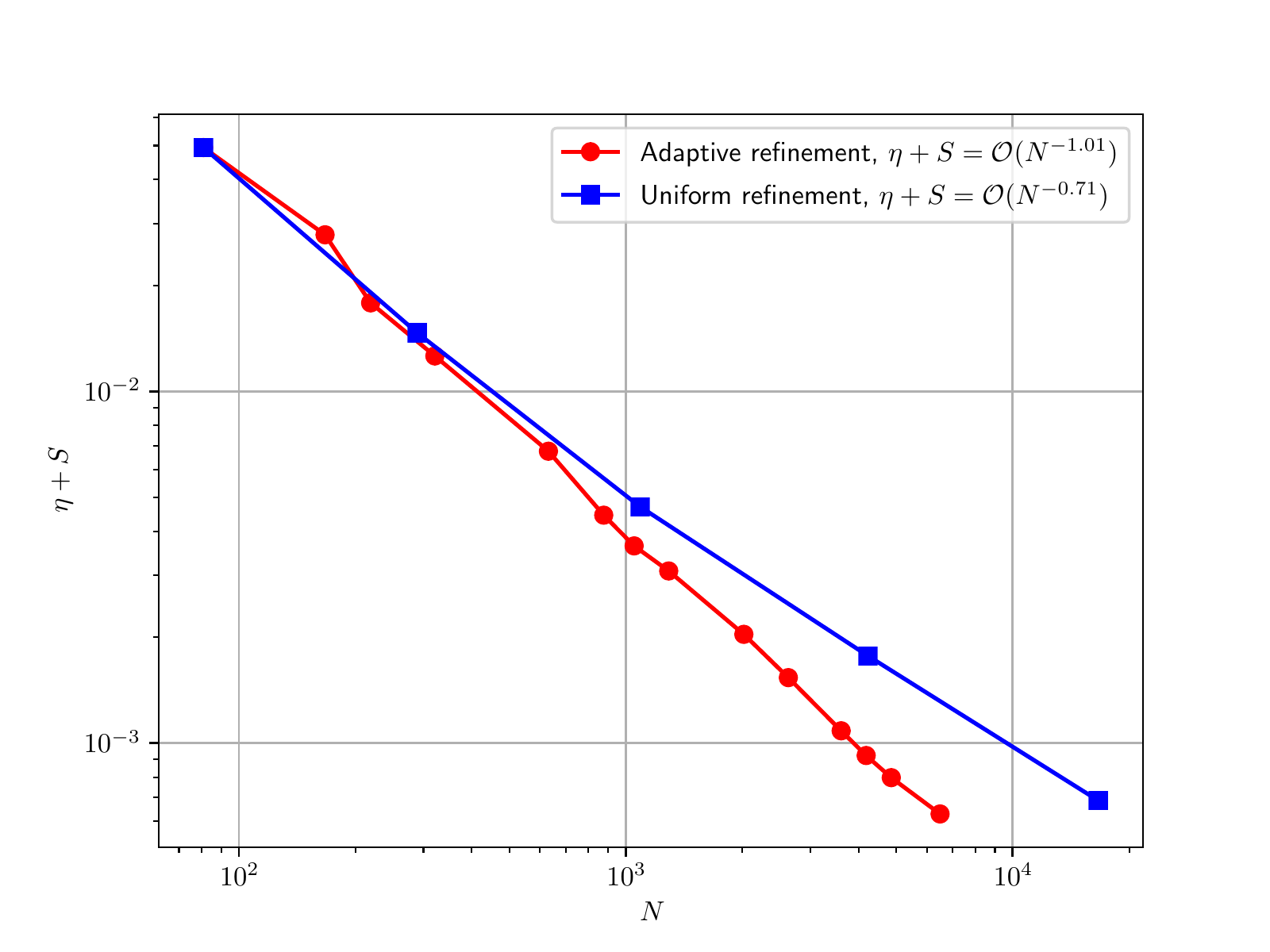}
    \caption{Convergence of the error estimator $\eta + S$
    as a function of the number of degrees of freedom $N$ for the
    uniform and the adaptive mesh sequences. The observed approximate
    convergence rates are given in the text box window.}
    \label{fig:convergence}
\end{figure}

\bibliographystyle{siam}
\bibliography{PS-ref}

\end{document}